\documentclass[12 pt, a4paper]{article}
\usepackage[pdftex]{graphicx}
\usepackage{amsmath}
\usepackage{mathtools}
\usepackage{amsthm}
\usepackage{amsfonts}
\usepackage{amssymb}
\usepackage[margin=1.0in]{geometry}
\usepackage{tikz}
\usetikzlibrary{cd}
\usepackage{enumitem}
\usepackage[osf]{mathpazo}

\newtheorem{theorem}{Theorem}

\newtheorem{lemma}[theorem]{Lemma}
\newtheorem{prop}[theorem]{Proposition}

\theoremstyle{definition}
\newtheorem{defin}[theorem]{Definition}
\newtheorem{fact}[theorem]{Fact}

\theoremstyle{remark}
\newtheorem*{rem}{Remark}

\renewcommand{\phi}{\varphi}

\newcommand{\aut}{\mathrm{Aut}}

\newcommand{\lip}{\mathrm{Lip}}

\begin{document}

\title{Minimal model-universal flows for locally compact Polish groups}
\author{Colin Jahel and Andy Zucker}

\maketitle
\begin{abstract}
Let $G$ be a locally compact Polish group. A metrizable $G$-flow $Y$ is called model-universal if by considering the various invariant probability measures on $Y$, we can recover every free action of $G$ on a standard Lebesgue space up to isomorphism. Weiss has shown that for countable $G$, there exists a minimal, model-universal flow. In this paper, we extend this result to all locally compact Polish groups.
\let\thefootnote\relax
\let\thefootnote\relax\footnote{A.\ Zucker was supported by NSF Grant no.\ DMS 1803489. C.\ Jahel and A.\ Zucker were partially supported by the ANR project AGRUME (ANR-17-CE40-0026).}
\end{abstract}

\section{Introduction}

Topological dynamics can be understood as an attempt to describe the actions of topological groups on compact spaces. Part of this process consists of understanding which objects are generic or universal among the actions of a given group. A famous instance of this is the universal minimal flow, whose existence was proven by Ellis (\cite{Ellis}). This is a minimal flow which maps onto any other minimal flow; by understanding the properties of this one object, we can better understand the collection of all minimal flows. In this paper, we prove the existence of a minimal flow which is universal in a different sense, in that it contains a copy of any measured free action. Similarly, this ``universal minimal model" can help shed light on the dynamical properties of a given locally compact group.

Let $G$ be a locally compact non-compact Polish group. A \emph{$G$-flow} is a continuous $G$-action on a compact space. A $G$-flow is said to be \emph{minimal} if every orbit is dense. If $Y$ is a $G$-flow, then $P_G(Y)$ denotes the $G$-invariant regular Borel probability measures on $Y$.

By a \emph{$G$-system}, we will mean a Borel $G$-action on a standard Lebesgue space $(X, \mu)$ which preserves $\mu$. We say that a $G$-system $(X, \mu)$ is \emph{free} if the set 
$$\mathrm{Free}(X):= \{x\in X: \forall g\in (G\setminus \{1_G\})\,\, gx\neq x\}$$
has measure $1$. Because $G$ is locally compact, this set is Borel. Therefore when dealing with free $G$-systems, we will often just assume that $X = \mathrm{Free}(X)$.  

If $(X, \mu)$ and $(Y, \nu)$ are $G$-systems, we say that $Y$ is a \emph{factor} of $(X, \mu)$ if there is a Borel, $G$-invariant subset $X'\subseteq X$ with $\mu(X') = 1$ and a Borel, $G$-equivariant map $f\colon X'\to Y$ with $\nu = f_*\mu$. 

If we can find $f$ as above which is also injective, then we say that $(X, \mu)$ and $(Y, \nu)$ are \emph{isomorphic} $G$-systems.  We will denote this $(X, \mu)\cong (Y, \nu)$.

A compact metric $G$-flow $Y$ is \emph{weakly model-universal} if for every free $G$-system $(X, \mu)$, there is $\nu\in P_G(Y)$ with $(Y, \nu)$ a factor of $(X, \mu)$. We say that $Y$ is \emph{model-universal} if for every free $G$-system $(X, \mu)$, there is $\nu\in P_G(Y)$ with $(X, \mu)\cong (Y, \nu)$.

The main theorem of this paper is the construction of a minimal, model-universal flow for every locally compact, non-compact Polish group.
\vspace{2 mm}
 
\begin{theorem}\label{Thm:Princip}
Let $G$ be a locally compact, non-compact Polish group. Then there exists a minimal model-universal flow for $G$.
\end{theorem}
\vspace{2 mm}

This result extends a result of Weiss \cite{W}, who proved Theorem~\ref{Thm:Princip} in the case of countable groups. Weiss uses the slightly different terminology \emph{universal minimal model} to describe this object. More recently, Zucker \cite{Z} gives a new proof of Weiss's result, and it is this proof that we generalize. We remark that Theorem~\ref{Thm:Princip} cannot extend to all Polish groups. As an example, the group $\aut(\mathbb{Q})$ admits no non-trivial minimal flows, while the shift action on $[0,1]^\mathbb{Q}$ equipped with the product Lebesgue measure is a free $G$-system. It would be interesting to understand the precise class of Polish groups for which Theorem~\ref{Thm:Princip} is true.

The universality of the flow from Theorem~\ref{Thm:Princip} is quite different from that of the universal minimal flow $M(G)$. Indeed, $M(G)$ is universal in the sense that it surjects onto every minimal flow, while we construct a minimal flow into which every free $G$-system can be injected. Another difference is uniqueness; while $M(G)$ is unique up to isomorphism, it is shown in \cite{Z} that when $G$ is countable, there are continuum many minimal model-universal flows up to isomorphism. Unfortunately, the proof of this requires some machinery for countable groups that is not yet known to generalize to locally compact groups. However, we strongly suspect that for any locally compact $G$, minimal model-universal flows are not unique.

Another result of \cite{Z} that we do not address here is whether a minimal, model-universal flow can be free. If $Y$ is a minimal flow and $\nu\in P_G(Y)$ is such that $(Y, \nu)$ is a free $G$-system, then $Y$ must be \emph{essentially free}, meaning that $\mathrm{Free}(Y)\subseteq Y$ is dense $G_\delta$. However, Weiss's construction of a minimal model-universal flow left open the question of whether such an object could be free. The first construction of a free minimal model-universal flow was given by Elek \cite{Elek}, and in \cite{Z}, an easy method of transforming any minimal model-universal flow into a free one is provided. The method is roughly as follows: start with $Z$ a minimal model-universal flow, where we note that the construction from \cite{Z} gives a zero-dimensional flow. Then construct an \emph{almost one-one} extension $\pi\colon Y\to Z$ with $Y$ free and so that for every $z\in \mathrm{Free}(Z)$, $|\pi^{-1}(\{z\})| = 1$. Then $\pi\colon \pi^{-1}(\mathrm{Free}(Z))\to \mathrm{Free}(Z)$ is a $G$-equivariant homeomorphism, and thus the map $\pi^{-1}\colon \mathrm{Free}(Z)\to Y$ will show that $Y$ is also model-universal. Unfortunately, it is essential that $Z$ be zero-dimensional for this to work, and our construction here does not produce zero-dimensional flows (indeed this is impossible when $G$ is connected). A simple example is provided by Antonyan \cite{Ant}; let $G = \mathbb{Z}/2\mathbb{Z}$ and set $Z = [-1, 1]^\omega$. $G$ acts on $Z$ by negating every coordinate. Hence $\mathrm{Free}(Z) = Z\setminus \{0\}$, where $0 = (0,0,0,...)$. However, Antonyan shows that $Z\setminus \{0\}$ does not embed as a $G$-subspace of any free $G$-flow. Therefore any soft method of transforming a minimal model-universal flow into a free one must use a different method to work for all locally compact groups.

One immediate consequence of Theorem~\ref{Thm:Princip} is a result that was suggested by Angel, Kechris and Lyons in \cite{AKL}. Recall that a topological group $G$ is said to be \emph{uniquely ergodic} if all minimal $G$-flows admit exactly one $G$-invariant probability measure.
\vspace{2 mm}

\begin{theorem}\label{Thm:nonUE}
Let $G$ be a locally compact non-compact Polish group.  Then there is a minimal $G$-flow with multiple invariant probability measures. In particular, $G$ is not uniquely ergodic. 
\end{theorem}

\begin{proof}
Let $Z$ be a minimal model-universal flow for $G$. It suffices to show that $Z$ admits an invariant probability measure that is not ergodic. Take $(X, \mu)$ a free $G$-system (see \cite{AEG}, Proposition 1.2). Then $G$ acts on $X\times 2$ by acting on the first coordinate. Letting $\delta_{1/2}$ be the $(1/2, 1/2)$-measure on $2$, then $(X\times 2, \mu\times \delta_{1/2})$ is a free, non-ergodic $G$-system. So letting $\nu\in P_G(Z)$ be chosen so that $(Z, \nu)\cong (X\times 2, \mu\times \delta_{1/2})$, we see that $\nu$ is not ergodic.
\end{proof}

\section{Preliminaries}
Fix a non-compact, locally compact Polish group $G$, on which we fix a right-invariant compatible metric $d$. Throughout we assume that $\{g\in G: d(g, 1_G)\leq 1\}$ is compact.
\vspace{2 mm}

\begin{defin}
	\label{Def:Lip}
	If $(X, d_X)$ and $(Y, d_Y)$ are compact metric spaces, $\lip(Y, X)$ denotes the space of Lipschitz functions from $Y$ to $X$ with the topology of pointwise convergence. If $f\in \lip(G, X)$ and $g, h\in G$, we set
	$$(g\cdot f)(h) = f(hg).$$
	This action turns $\lip(G, X)$ into a $G$-flow. Given a subflow $Y\subseteq \lip(G, X)$ and any $B\subseteq G$, we set $$S_B(Y) := \{y|_B: y\in Y\}.$$
	Notice that $S_B(Y)\subseteq \lip(B, X)$ is a compact metric space; when $B\subseteq G$ is pre-compact, we will use the uniform metric 
	$$d_B(u, v):= \sup\{d_X(u(g), v(g)): g\in B\}.$$
	If $u\in S_B(Y)$ and $g\in G$, we define $g\cdot u \in S_{Bg^{-1}}(Y)$ via $(g\cdot u)(h) = u(hg)$. If $A\subseteq G$ is another subset, we set 
	$$A|B := \{g\in G: A\subseteq Bg^{-1}\}.$$
	In particular, if $g\in A|B$ and $u\in S_B(Y)$, then $(g\cdot u)|_A\in S_A(Y)$.
\end{defin}

\begin{rem}
	If $(X, d_X)$ and $(Y, d_Y)$ are metric spaces, we will equip $X\times Y$ with the metric $d_{X\times Y}((x_0, y_0), (x_1, y_1)) = \max\{d_X(x_0, x_1), d_Y(y_0, y_1)\}$. 
\end{rem}
\vspace{2 mm}
	
We now spend some time understanding minimality in subflows of $\lip(G, X)$.
\vspace{2 mm} 
	
\begin{defin}
	\label{Def:MinimalFragment}
	Let $(X, d_X)$ be a compact metric space. If $\epsilon > 0$ and $S\subseteq X$, we say that $S\subseteq X$ is \emph{$\epsilon$-dense} if for every $x\in X$, there is $s\in S$ with $d_X(x, s)< \epsilon$. 
		
	Now fix a pre-compact $D\subseteq G$ and $\epsilon > 0$. We say that a subflow $Y\subseteq \lip(G, X)$ is \emph{$(D, \epsilon)$-minimal} if for any $y\in Y$, we have that $\{(g\cdot y)|_D: g\in G\}\subseteq S_D(Y)$ is $\epsilon$-dense. 
\end{defin} 
\vspace{2 mm}
	
We remark that this notion is monotone; if $Y$ is $(D, \epsilon)$-minimal, then also $Y$ is $(D_0, \epsilon_0)$-minimal for any $D_0\subseteq D$ and $\epsilon_0\geq \epsilon$.
\vspace{2 mm}
	
\begin{prop}
	\label{Prop:MinimalFragment}
	With notation as in Definition~\ref{Def:MinimalFragment}, the following are equivalent.
	\begin{enumerate}
		\item 
		The subflow $Y\subseteq \lip(G, X)$ is $(D, \epsilon)$-minimal.
		\item 
		There is a pre-compact open $E\subseteq G$ so that for any $u\in S_E(Y)$, we have that \newline $\{(g\cdot u)|_D: g\in D|E\}\subseteq S_D(Y)$ is $\epsilon$-dense. 
	\end{enumerate}
\end{prop}
	
\begin{proof}
	Item $2$ certainly implies item $1$. Conversely, if Item $2$ fails, let $E_n\subseteq G$ be an exhaustion of $G$ by pre-compact open sets, and find $u_n\in S_{E_n}(Y)$ and $v_n\in S_D(Y)$ with $d_D((g\cdot u_n)|_D, v_n)\geq \epsilon$ for every $g\in D|E_n$. Passing to a subsequence, let $u_n\to y\in Y$ and $v_n\to v\in S_D(Y)$. Towards a contradiction, suppose there were $g\in G$ with $d_D((g\cdot y)|_D, v)< \epsilon$. Then for some $N< \omega$, we have $g\in D|E_n$ for every $n\geq N$. But then we must have $d_D((g\cdot u_n)|_D, v_n)< \epsilon$ for some $n$, a contradiction. Hence $Y$ is not $(D, \epsilon)$-minimal.  
\end{proof}
\vspace{2 mm}

\begin{prop}
	\label{Prop:MinimalGDelta}
	With notation as in Definition~\ref{Def:MinimalFragment}, the subflow $Y\subseteq \lip(G, X)$ is minimal iff it is $(D, \epsilon)$-minimal for every pre-compact $D\subseteq G$ and $\epsilon > 0$. 
\end{prop}
	
\begin{proof}
	If $Y$ is minimal, then fix pre-compact $D\subseteq G$, $\epsilon > 0$, and $y\in Y$. If $u\in S_D(Y)$, find $z\in Y$ with $z|_D = u$. Find $g_n\in G$ with $g_n\cdot y\to z$. It follows that we must have $(g_n\cdot y)|_D \to u$, implying that $d_D((g_n\cdot y)|_D, u)< \epsilon$ for some $n< \omega$. Hence $Y$ is $(D, \epsilon)$-minimal.
		
	Now assume that $Y$ is $(D, \epsilon)$-minimal for every pre-compact $D\subseteq G$ and $\epsilon > 0$. Let $E_n\subseteq G$ be an exhaustion of $G$ by pre-compact open sets, and let $\epsilon_n\to 0$. Fix $y, z\in Y$. We can then find for each $n< \omega$ some $g_n\in G$ with $d_{E_n}((g_n\cdot y)|_{E_n}, z|_{E_n}) < \epsilon_n$. It follows that $g_n\cdot y\to z$, showing that $Y$ is minimal.
\end{proof}
\vspace{2 mm}

The remainder of this section looks at other important flows closely related to $\lip(G, X)$.
\vspace{2 mm}
	
\begin{defin}
	\label{Def:ClosedSetsG}
We let $2^G$ denote the space of closed subsets of $G$ with the Fell topology. If $B\subseteq G$, we set $\mathrm{Meets}(B) := \{F\in 2^G: F\cap B \neq \emptyset\}$ and $\mathrm{Avoids}(B) := \{F\in 2^G: F\cap B = \emptyset\}$. By considering the sets $\mathrm{Meets}(U)$ and $\mathrm{Avoids}(C)$ for $U\subseteq G$ open and $C\subseteq G$ compact, we obtain a sub-basis for the Fell topology. The action we take is not quite the obvious one: given $g\in G$ and $F\in 2^G$, we set $g\boldsymbol{\cdot} F = Fg^{-1}$. We do this for the following reason: letting $\iota\colon 2^G\to \lip(G, [0,1])$ denote the map $\iota(S)(g) = d(g, S)$, then $\iota$ is an injective $G$-map.
\end{defin}
\vspace{2 mm}

\begin{defin}
	\label{Def:MaxSpacedSets}
	Fix $D\subseteq G$ any symmetric subset. Then a set $S\subseteq G$ is \emph{$D$-spaced} if $Dg\cap Dh = \emptyset$ for any $g\neq h\in S$. Notice that $S\subseteq G$ is a maximal $D$-spaced set iff $S$ is $D$-spaced and $D^2S = G$. 
	
	Now suppose $D\subseteq G$ is a compact symmetric neighborhood of the identity. We let $G_D\subseteq 2^G$ denote the collection of $S\subseteq G$ which are $\mathrm{Int}(D)$-spaced and with $D^2S = G$. In particular, $G_D$ contains every maximal $D$-spaced set and every maximal $\mathrm{Int}(D)$-spaced set.
\end{defin}

\begin{prop}
	\label{Prop:GDWeak}
	For every $D\subseteq G$ a compact symmetric neighborhood of the identity, the $G$-flow $G_D$ is weakly model-universal. 
\end{prop}

\begin{proof}
This fact follows from a theorem of Kechris and a theorem of Slutsky. Suppose $G\times X\to X$ is a standard Borel $G$-space. A \emph{cross-section} is any Borel $C\subseteq X$ so that $G\cdot C = X$.  If $D\subseteq G$ is a compact neighborhood of the identity, we say that a cross-section $C$ is \emph{$D$-lacunary} if whenever $x\neq y\in C$, we have $Dx\cap Dy = \emptyset$. A \emph{maximal $D$-lacunary cross-section} is a $D$-lacunary cross-section $C$ such that for any $x\not\in C$, $C\cup\{x\}$ is not $D$-lacunary. Notice that if $C$ is a maximal $D$-lacunary cross-section, then $D^2C = X$.

Kechris in \cite{Kechris} proves the following.
\vspace{2 mm}

\begin{theorem}
	\label{Thm:KechrisLacunary}
	Let $G$ be a locally compact group and $D$ a compact neighborhood of the identity. Then any standard Borel $G$-space $X$ admits a $D$-lacunary cross section.
\end{theorem}
\vspace{2 mm}
 
The measurable version of this result is a classical result of Feldman, Hahn, and Moore \cite{FHM}; however, we appeal to the Borel version as we will need the following strengthening due to Slutsky (\cite{Slutsky}, Theorem 2.4).
\vspace{2 mm}

\begin{theorem} 
	\label{Thm:SlutskyMaxLacunary}
	Let $G$ be a locally compact group, and let $X$ be a standard Borel $G$-space. Then if $D$ is a compact neighborhood of the identity and $C\subseteq X$ is a $D$-lacunary cross-section of $X$, there is a maximal $D$-lacunary cross section $C'\supseteq C$.
\end{theorem}
\vspace{2 mm}

Now fix $(X, \mu)$ a free $G$-system. Applying Theorem~\ref{Thm:SlutskyMaxLacunary}, let $C\subseteq X$ be a maximal $D$-lacunary cross section. We define a Borel $G$-map $f: X\to G_D$ by setting 
$$f(x) = C_x:= \{g\in G: gx\in C\}.$$
As $C$ is a maximal $D$-lacunary cross-section, we have $C_x\in G_D$, and certainly $f$ is $G$-equivariant. To see that $f$ is Borel, suppose $B\subseteq G$ is Borel, and consider $\mathrm{Meets}(B)\subseteq 2^G$. Then $f(x)\in \mathrm{Meets}(B)$ iff $Bx\cap C\neq \emptyset$. The set $Y:= \{(g, x): g\in B \text{ and } gx\in C\}\subseteq G\times X$ is Borel, and $f^{-1}(\mathrm{Meets}(B)) = \pi_X[Y]$. We note that since $C$ is $D$-lacunary, the projection $\pi_X$ is countable-to-one on $Y$, showing that $f^{-1}(\mathrm{Meets}(B))$ is Borel.
\end{proof}

\section{A (non minimal) model-universal flow}

In this section, we prove
\vspace{2 mm}

\begin{theorem} \label{Thm:BetterVarad}
	The $G$-flow $\lip(G, [0,1])^\omega$ is model-universal.
\end{theorem}
\vspace{2 mm}

The proof is an adaptation of the proof of a result of Varadarajan.
\vspace{2 mm}

\begin{theorem}[Varadarajan \cite{V}]
	\label{Thm:Varad}
	Let $G$ be a locally compact Polish group, and let $G\times X\to X$ be a Borel action of $G$ on a standard Borel space $X$. Then there is a compact metric $G$-flow $Y$ and a $G$-equivariant Borel injection $X\hookrightarrow Y$.
\end{theorem}
\vspace{2 mm}

Write $\lip := \lip(G, [0,1])$, and let $\lambda$ denote the left Haar measure on $G$. 
\vspace{2 mm}

\begin{defin}
	\label{Def:Convolution}
	Suppose $f\colon X\to [0,1]$ is Borel. We let $f_x\colon G\to [0,1]$ be the Borel function given by $f_x(g) := f(gx)$. If $\phi\colon G\to \mathbb{R}^+$ is in $L^1(G, \lambda)$ with $\|\phi\|_1 \leq 1$, we let $(\phi*f)_x\colon G\to [0,1]$ be defined via
	$$(\phi*f)_x(g) = \int_G \phi(h)f(h^{-1}gx) d\lambda(h)$$
\end{defin} 
\vspace{2 mm}	
	
\begin{lemma}
	\label{Lemma:LipschitzConv}
	With notation as in Definition~\ref{Def:Convolution}, suppose $\lambda(\mathrm{Supp}(\phi)) \leq L$ and that $\phi$ is $K$-Lipschitz. Then $(\phi*f)_x$ is $2L\cdot K\cdot \|\phi\|_1$-Lipschitz. 
\end{lemma}

\begin{proof}
	Fix $g_0, g_1\in G$. By considering the change of variables $h\to g_ih$, we see that
	$$\left|(\phi*f)_x(g_0) - (\phi*f)_x(g_1)\right| \leq \int_G \left|\phi(g_0h)-\phi(g_1h)\right|f(h^{-1}x)d\lambda(h)$$
	The right hand side is identically zero whenever $h\not\in g_0^{-1}\cdot\mathrm{Supp}(\phi)\cup g_1^{-1}\cdot \mathrm{Supp}(\phi)$. For $h$ inside this set, the integrand is at most $K\cdot \|\phi\|_1\cdot d(g_0, g_1)$.
\end{proof}
\vspace{2 mm}

In particular, suppose $\phi$ is such that $(\phi*f)_x\in \lip$ for every $x\in X$. Then the map $\phi*f\colon X\to \lip$ given by $\phi*f(x) = (\phi*f)_x$ is Borel and $G$-equivariant.

Recall that a sequence $(\phi_n)_n$ from $L^1(G, \lambda)$ is an \emph{approximate identity} if $\phi_n \geq 0$, $\phi_n$ is symmetric, $\|\phi_n\|_1 = 1$, $\mathrm{Supp}(\phi_n)$ is compact, and $\mathrm{Supp}(\phi_n)\to \{1_G\}$.
\vspace{2 mm}

\begin{fact}[\cite{F}, Proposition 2.44]
	\label{Fact:ApproxIdDistinguish}
	Suppose $f\colon [0,1]\to G$ is Borel, and let $(\phi_n)_n$ be an approximate identity. Then for any $x\in X$ and any compact $K\subseteq G$, we have that $\|(\phi_n*f)_x\cdot \chi_K - f_x\cdot \chi_K\|_1\to 0$.
\end{fact}
\vspace{2 mm}

We can now work towards our proof of Theorem~\ref{Thm:BetterVarad}. Let $\{f_k: k< \omega\}$ be a sequence of characteristic functions of Borel subsets of $X$ which generate $X$. Let $(\phi_n)_n$ be an approximate identity with each $\phi_n$ $C_n$-Lipschitz for some $C_n\in \mathbb{R}^+$. Using Lemma~\ref{Lemma:LipschitzConv}, choose constants $c_n > 0$ so that $(c_n\phi_n*f_k)_x\in \lip$ for every $n, k< \omega$ and $x\in X$. We define the map $\gamma\colon X\to \lip^{\omega\times \omega}\cong \lip^\omega$ by setting $\gamma(x)(n, k) = (c_n\phi_n*f_k)_x$. Then $\gamma$ is Borel and $G$-equivariant, and we need only check that it is injective. Suppose $\gamma(x) = \gamma(y)$. Using Fact~\ref{Fact:ApproxIdDistinguish}, this implies that for each $k< \omega$, $(f_k)_x(g) = (f_k)_y(g)$ for $\lambda$-almost every $g\in G$. So for most $g\in G$, this is true for every $k< \omega$. Fix such a $g\in G$. But then $gx = gy$, since the $f_k$ separate points. In particular, also $x = y$.

\section{A minimal model-universal flow}

In this section we prove our main theorem.

\begin{defin}
	Suppose $D\subseteq G$ is a compact symmetric neighborhood of the identity. We call $S_0, S_1\subseteq G$ \emph{$D$-apart} if $DS_0\cap DS_1 = \emptyset$. 
	
	Suppose $(X, d_X)$ is a compact metric space. We call a subflow $Y\subseteq \lip(G, X)$ \emph{$D$-irreducible} if whenever $S_0, S_1\subseteq G$ are $D$-apart and whenever $y_0, y_1\in Y$, there is $z\in Y$ with $z|_{S_i} = y_i|_{S_i}$.
\end{defin}
\vspace{2 mm}
	
We note that if $d_X$ has diameter $1$ and $D\supseteq \{g\in G: d(g, 1_G) \leq 1\}$, then $\lip(G, X)$ is $D$-irreducible.
\vspace{2 mm}

\begin{defin}
	\label{Def:Theta}
	Suppose $(X, d_X)$ is a compact metric space of diameter $1$, suppose $D\supseteq \{g\in G: d(g, 1_G)\leq 1\}$ is a compact symmetric neighborhood of $1_G$, and let $Y\subseteq \lip(G, X)$ be a $D$-irreducible subflow. 
	
	Fix $E\subseteq G$ another compact symmetric neighborhood of $1_G$ with $D^2\subseteq E$. Fix some $u\in S_E(Y)$. Suppose $F\subseteq G$ is yet another compact symmetric neighborhood of $1_G$ with $E^3\subseteq F$. We define the flow
	$$\Theta(Y, u, F)\subseteq \lip(G, X)$$
	to consist of those functions $f\in \lip(G, X)$ so that all of the following hold:
	\begin{enumerate}
		\item
		There is $T\in G_F$ so that $(g\cdot f)|_E = u$ for each $g\in T$.
		\item 
		There is $y\in Y$ with $f(g) = y(g)$ for any $g\not\in E^2T$.
		\item 
		For every $g\in G$, we have $(g\cdot f)|_D \in S_D(Y)$.   
	\end{enumerate}
\end{defin}	
\vspace{2 mm}

If $f_n\in \Theta(Y, u, F)$, where items $1$ and $2$ are witnessed by $T_n\in G_F$ and $y_n\in Y$, respectively, then suppose $f_n\to f\in \lip(G, X)$. To show that $f\in \Theta(Y, u, F)$, we first note that item $3$ is a closed condition. Then pass to a subsequence with $T_n\to T\in G_F$ and $y_n\to y\in Y$. Then $T$ and $y$ will witness that items $1$ and $2$ hold for $f$, showing that $\Theta(Y, u, F)$ is closed. To see that $\Theta(Y, u, F)$ is $G$-invariant, take $f\in \Theta(Y, u, F)$ and $g\in G$. Item $3$ is clear for $g\cdot f$, and if $T\in G_F$ and $y\in Y$ witness items $1$ and $2$ for $f$, then $g\boldsymbol{\cdot} T = Tg^{-1}$ and $g\cdot y$ will be witnesses for $g\cdot f$.

We remark that since $Y$ is $D$-irreducible, the flow $\Theta(Y, u, F)$ is non-empty; the next proposition shows this and more.
\vspace{2 mm}

\begin{prop}
	\label{Prop:Uniformize}
	In the setting of Definition~\ref{Def:Theta}, suppose in addition that $Y$ is weakly model-universal. Then so is $\Theta(Y, u, F)$.
\end{prop}

\begin{proof}
	First consider the restriction map $S_{E^3}(Y)\to S_{E^3\setminus E^2}\times S_E(Y)$. Since $Y$ is $D$-irreducible, this map is surjective. As this is a continuous surjection between compact metric spaces, let $\eta\colon S_{E^3\setminus E^2}(Y)\times S_E(Y) \to S_{E^3}(Y)$ be a Borel section.
	
	We now define a map $\theta\colon Y\times G_F\to \Theta(Y, u, F)$, where given $y\in Y$ and $T\in G_F$, $\theta(y, T)$ is defined as follows.
	\begin{enumerate}
		\item 
		If $g\not\in E^2T$, then $\theta(y, T)(g) = y(g)$.
		\item 
		If $g = hk$ for some $h\in E^3$ and $k\in T$, we set $\theta(y, T)(g) = \eta((k\cdot y)|_{E^3\setminus E^2}, u)(h)$.
	\end{enumerate}

Then $\theta$ is a Borel $G$-equivariant map. We check that $\theta(y, T)\in \Theta(Y, u, F)$. By construction, items 1 and 2 from Definition~\ref{Def:Theta} are satisfied. For item 3, consider $g\in G$. Since $D^2\subseteq E$, either $Dg\cap E^2T = \emptyset$ or $Dg\subseteq E^3k$ for some $k\in T$. If $Dg\cap E^2T = \emptyset$, then $(g\cdot \theta(y, T))|_D = (g\cdot y)|_D\in S_D(Y)$. If $D_g\subseteq E^3k$ for some $k\in T$, then 
$$\left(g\cdot \theta(y, T)\right)|_D = \left(gk^{-1}\cdot \eta\left((k\cdot y)|_{E^3}, u\right)\right)\bigg|_D \in S_D(Y).$$  
Lastly, to see that $\theta(y, T)\in \lip(G, X)$, we note that by assumption $D\supseteq \{g\in G: d(g, 1_G)\leq 1\}$, and we have just seen that item 3 holds.

Since $Y\times G_F$ is weakly model-universal, then so is $\Theta(Y, u, F)$.
\end{proof}
\vspace{2 mm}

We will sometimes want to refer to the $\theta$ constructed here as the various parameters change. In this case, we will refer to the map as $\theta\langle Y, u, F\rangle$.
\vspace{2 mm}

\begin{prop}
	\label{Prop:Irreducible}
	In the setting of Definition~\ref{Def:Theta}, $\Theta(Y, u, F)$ is $F^6$-irreducible.
\end{prop}

\begin{proof}
	Let $f_0, f_1\in \Theta(Y, u, F)$, where the membership of $f_i$ is witnessed by $T_i\in G_F$ and $y_i\in Y$. Suppose $S_0, S_1\subseteq G$ are $F^6$-apart. Set $T = (F^5S_0\cap T_0)\cup (F^5S_1\cap T_1)$. Then $T$ is an $\mathrm{Int}(F)$-spaced set, so let $U\supseteq T$ be a maximal $\mathrm{Int}(F)$-spaced set, and let $V = (U\setminus F(S_0\cup S_1))\cup T$. So in particular, $V\cap FS_i = T_i\cap FS_i$. 
	
	We claim that $V\in G_F$. Since $V\subseteq U$, $V$ is $\mathrm{Int}(F)$-spaced. To see that $F^2V = G$, let $g\in G$. If $g\not\in F^3(S_0\cup S_1)$, then $g\in \mathrm{Int}(F)^2h$ for some $h\in U$, and since $h\not\in F(S_0\cup S_1)$, we have $h\in V$. If $g\in F^3S_i$, then $g\in F^2h$ for some $h\in (T_i\cap F^5S_i) \subseteq T\subseteq V$. 
	
	Since $Y$ is $D$-irreducible and $D\subseteq F$, let $y\in Y$ be chosen with $y|_{F^5S_i} = y_i|_{F^5S_i}$. Using $V$ and $y$, we define the $f\in \Theta(Y, u, F)$ which will satisfy $f|_{S_i} = f_i|_{S_i}$. We set $(g\cdot f)|_E = u$ for each $g\in V$, and we set $f(g) = y(g)$ whenever $g\not\in E^2V$. It remains to define $f$ on $(E^2\setminus E)g$ for $g\in V$. If $g\in V$ and $E^2g\cap S_i \neq \emptyset$, we set $f|_{E^2g} = f_i|_{E^2g}$. If $g\in V$ and $E^2g\cap (S_0\cup S_1) = \emptyset$, then we use the $D$-irreducibility of $Y$ to find any $v_g\in S_{E^3}(Y)$ with $v_g|_{E^3\setminus E^2} = (g\cdot f)|_{E^3\setminus E^2}$ and $v_g|_E = u$, and we set $(g\cdot f)(h) = v_g(h)$ for $h\in E^3$.   
	
	We verify that $f$ is as desired. Since $V\cap FS_i = T_i\cap FS_i$, we have $f|_{S_i} = f_i|_{S_i}$. To check that $f\in \Theta(Y, u, F)$, items $1$ and $2$ are witnessed by $V$ and $y$. For item $3$, let $g\in G$. Then either $Dg\cap E^2V = \emptyset$, in which case $(g\cdot f)|_E = (g\cdot y)|_E\in S_E(Y)$, or $Dg\subseteq E^3h$ for some $h\in V$, in which case $(g\cdot f)|_E = (gh^{-1})\cdot v_h|_E\in S_E(Y)$.
\end{proof}

\vspace{4 mm}

We can now undertake the main construction. We set $\lip = \lip(G, [0,1])$. So note in particular that $\lip^n = \lip(G, [0,1]^n)$, and we can treat $\lip^0$ as the trivial (singleton) $G$-flow. We will inductively construct the following objects:
\begin{itemize}
	\item 
	$D_n, E_n, F_n\subseteq G$ compact symmetric neighborhoods of the identity,
	\item 
	A subflow $Y_n\subseteq \lip^n$ so that $Y_n\times \lip\subseteq \lip^{n+1}$ is $D_n$-irreducible.
\end{itemize}

The sets $D_n, E_n, F_n\subseteq G$ will have the following properties:

\begin{enumerate}
	\item 
	$D_n^2\subseteq E_n$, $E_n^3\subseteq F_n$, $F_n^6\subseteq D_{n+1}$
	\item 
	$\bigcup_n D_n = G$
	\item 
	Suppose $K_n< \omega$ is such that $S_{D_n}(Y_n\times \lip)$ can be covered by $K_n$-many balls of radius $1/2^n$. Then there is a $D_n^2$-spaced set $\{g_{n,i}: i< K_n\}\subseteq G$ so that $D_ng_{n,i}\subseteq E_n$ for each $i< K_n$.
	\item 
	There is a $E_n^4$-spaced set $\{h_{n,i}: i< 2^n\}\subseteq G$ so that $E_n^4h_{n,i}\subseteq \mathrm{Int}(F_n)$ for each $i< 2^n$. 
\end{enumerate}

We remark that the construction of the sets $D_n, E_n, F_n\subseteq G$ requires that $G$ be non-compact, especially in regards to items 3 and 4. Start by setting $D_0 = \{g\in G: d(g, 1_G)\leq 1\}$ and $Y_0 = \lip^0$. 

Suppose we have defined $Y_0,...,Y_n$; $D_0,...,D_n$; $E_0,...,E_{n-1}$; and $F_0,...,F_{n-1}$. In particular, we have arranged so far that $Y_n\times \lip\subseteq \lip^{n+1}$ is $D_n$-irreducible. Find $E_n\subseteq G$ satisfying items $(1)$ and $(3)$. 
\vspace{2 mm}

\begin{lemma}
	\label{Lem:ConstructString}
	There is $u_n\in S_{E_n}(Y_n\times \lip)$ so that $$\{(g\cdot u_n)|_{D_n}\colon g\in D_n|E_n\}\subseteq S_{D_n}(Y_n\times \lip)$$
	is $(1/2^n)$-dense.
\end{lemma}

\begin{proof}
	By assumption, $Y_n\times \lip$ is $D_n$-irreducible. Suppose $K_n< \omega$ is as in item 3, and let $\{f_i: i< K_n\}\subseteq S_{D_n}(Y_n\times \lip)$ be chosen so that every $f\in S_{D_n}(Y_n\times \lip)$ satisfies $d_{D_n}(f, f_i)< 1/2^n$ for some $i< K_n$. Let $\{g_{n,i}: i< K_n\}\subseteq G$ be as guaranteed by item 3. Then $Dg_{n,i}$ and $Dg_{n, j}$ are $D_n$-apart whenever $i\neq j < K_n$. Using $D_n$-irreducibility, we can find $u_n\in S_{E_n}(Y_n\times \lip)$ so that $(g_{n,i}\cdot u_n)|_{D_n} = f_i$ for every $i< K_n$. Then $u_n$ is as desired.
\end{proof}
\vspace{2 mm}

Now find $F_n\subseteq G$ satisfying items $(1)$ and $(4)$, and form 
$$Y_{n+1} := \Theta(Y_n\times L_1, u_n, F_n)\subseteq L_{n+1}.$$ 
Then find $D_{n+1}\subseteq G$ satisfying item $(1)$. By Proposition~\ref{Prop:Irreducible}, $Y_{n+1}\times \lip$ is $D_{n+1}$-irreducible.

We can regard each $Y_n$ as a subflow of $\lip^\omega$ by adding a tail of constant zero functions. Conversely, if $m< n\leq \omega$, we let $\pi^n_m\colon [0,1]^n\to [0,1]^m$ denote projection onto the first $m$ coordinates. So if $Z\subseteq \lip^\omega$ is a subflow, we let $\pi^\omega_n\circ Z$ denote its projection to a subflow of $\lip^n$. We now consider the space $\mathrm{Sub}(\lip^\omega)$ of subflows of $\lip^\omega$ equipped with the Vietoris topology. In this topology, we have $Z_n\to Z$ iff for each compact $D\subseteq G$ and $m< \omega$, we have $S_D(\pi^n_m\circ Z_n)\to^n S_D(\pi^\omega_m\circ Z)$ in the space $K(\lip(D, [0,1]^m))$, the space of compact subsets of $\lip(D, [0,1]^m)$ also equipped with the Vietoris topology.
\vspace{2 mm}

\begin{lemma}
	\label{Lem:ConvergeToMinimal}
	In the space $\mathrm{Sub}(\lip^\omega)$, we have $Y_n\to Z$ for some minimal flow $Z$.
\end{lemma}

\begin{proof}
	For notation, set $S(k, m, n) := S_{D_k}(\pi^n_m\circ Y_n)$. Notice that when $k\leq n$, we have that $S(k, m, n+1)\subseteq S(k, m, n)$ by item $3$ in Definition~\ref{Def:Theta}. So it follows that $S(k, m, n)\to^n S(k, m)$ for some compact $S(k, m)\subseteq \lip(D_k, [0,1]^m)$. Also notice that $\pi^{m+1}_m\circ S(k, m+1) = S(k, m)$. Furthermore, if $\ell \geq k$ and $m< \omega$, let 
	$$\rho^l_k\colon \lip(D_\ell, [0,1]^m)\to \lip(D_k, [0,1]^m)$$
	denote the restriction map. We note that whenever $k \leq \ell \leq n$, we have $\rho^l_k[S(\ell, m, n)] = S(k, m, n)$, so also $\rho^l_k[S(\ell, m)] = S(k, m)$. It follows that 
	$$Z := \varprojlim_k \varprojlim_m S(k, m)\subseteq \lip^\omega$$
	satisfies $Y_n\to Z$ in $K(\lip^\omega)$. Since $\mathrm{Sub}(\lip^\omega)$ is a closed subspace of $K(\lip^\omega)$, it follows that $Z\in \mathrm{Sub}(\lip^\omega)$. 
	
	To see that $Z$ is minimal, it suffices to argue that $\pi^\omega_m\circ Z$ is minimal for each $m< \omega$. To do this, we use Proposition~\ref{Prop:MinimalGDelta}. Fix $z\in \pi^\omega_m\circ Z$ and $k< \omega$. We will argue that $\{(g\cdot z)|_{D_k}: g\in G\}$ is dense in $S(k, m)$, thus handling all $\epsilon > 0$ simultaneously.
	
	 Let $n\geq k$. The construction of $Y_{n+1}$ yields (keeping in mind that $F_n^6\subseteq D_{n+1}$) that for any $v\in S(n+1, m, n+1)$, we have that $\{(g\cdot v)|_{D_n}: g\in D_n|D_{n+1}\}\subseteq S(n, m, n)$ is $(1/2^n)$-dense. Since $z|_{D_{n+1}}\in S(n+1, m)$, we have that $\{(g\cdot z)|_{D_n}: g\in D_n|D_{n+1}\}\subseteq S(n, m)$ is $(1/2^n)$-dense. By restricting to $D_k$, we see that $\{(g\cdot z)|_{D_k}: g\in D_n|D_{n+1}\}\subseteq S(k, m)$ is also $(1/2^n)$-dense. Letting $n$ grow, we see that $\{(g\cdot z)|_{D_k}: g\in G\}\subseteq S(k, m)$ is dense as desired. 
\end{proof}
\vspace{2 mm}

Set $X = \lip^\omega\times \prod_{n} G_{F_n}$. Weakly model-universal flows are closed under products, and if any member of the product is model-universal, then the product is as well (\cite{Z}, Proposition 6). So $X$ is model-universal. We will often write elements of $X$ as tuples $((f_n)_n, (S_n)_n))$, where $f_n\in \lip$ and $S_n\in G_{F_n}$. We will find a Borel $G$-invariant set $W\subseteq X$ with the following two properties:
\begin{enumerate}
	\item 
	For every measure $\mu\in P_G(X)$, $\mu(W) = 1$.
	\item 
	There is a Borel, $G$-equivariant injection $\phi\colon W\to Z$.
\end{enumerate}  
This will show that $Z$ is model-universal. We set
$$W = \{((f_n)_n, (S_n)_n): \forall k< \omega \,\, \exists m< \omega\,\, \forall n\geq m\,\, (E_k\cap E_n^3S_n = \emptyset)\}.$$
In particular, membership in $W$ only depends on the sets $S_n$. Fix $\mu\in P_G(X)$; we will show that $\mu(W) = 1$. Fix $n< \omega$, and let $\nu\in P_G(G_{F_n})$ be the projection of $\mu$ onto $G_{F_n}$. 
\vspace{2 mm}

\begin{lemma}
	\label{Lem:BorelCantelli}
If $k\leq n$, we have
$$\nu(\{S\in G_{F_n}: E_k\cap E_n^3S \neq \emptyset\}) \leq 1/2^n$$
\end{lemma}

\begin{proof}
	Recall that by item $(4)$ of the properties of the sets $D_n, E_n, F_n$, there is an $E_n^4$-spaced set $\{h_{n,i}: i< 2^n\}\subseteq G$ so that $E_n^4h_{n,i}\subseteq \mathrm{Int}(F_n)$ for each $i< 2^n$. Note that we have:
	\begin{align*}
	h_{n,i}^{-1}\boldsymbol{\cdot} \{S\in G_{F_n}: E_k\cap E^3_nS\neq \emptyset\} &= \{Sh_{n,i}\in G_{F_n}: E_k\cap E_n^3S\neq \emptyset\} \\
	&= \{T\in G_{F_n}: E_n^3E_kh_{n,i}\cap T\neq \emptyset\}\\
	&\subseteq \{T\in G_{F_n}: E_n^4h_{n,i}\cap T\neq \emptyset\}.
	\end{align*}
	Note that if $g_i\in E_n^4h_{n,i}\subseteq \mathrm{Int}(F_n)$ and $g_j\in E_n^4h_{n,j}\subseteq \mathrm{Int}(F_n)$, then we have 
	$$1_G\in \mathrm{Int}(F_n)g_i\cap \mathrm{Int}(F_n)g_j\neq \emptyset.$$
	It follows that the collection
	$$\{h_{n,i}^{-1}\boldsymbol{\cdot} \{S\in G_{F_n}: E_k\cap E_n^3S\}: i< 2^n\}$$
	is pairwise disjoint. Since $\nu$ is $G$-invariant, we are done.
\end{proof}
\vspace{2 mm}

We can now apply the Borel-Cantelli lemma to conclude that $\mu(W) = 1$.

We now turn to defining $\phi\colon W\to Z$. First let $j\colon \omega\to \omega\times 2$ be an infinite-to-one surjection. We define for each $n< \omega$ a Borel $G$-equivariant map $\phi_n\colon W\to Y_n$ so that $\phi(w) = \lim_n \phi_n(w)$. Start by letting $\phi_0$ denote the only map to $Y_0$. Suppose $\phi_n$ is defined. Let $w = ((f_n)_n, (S_n)_n)\in W$ be given. For notation, we set $f_{(n, 0)} = f_n$ and $f_{(n, 1)} = \iota(S_n)$. We set 
$$\theta_n := \theta\langle Y_n\times \lip, u_n, F_n\rangle$$
i.e.\ $\theta_n\colon (Y_n\times \lip)\times G_{F_n}\to Y_{n+1}$ denotes the map defined in the proof of Proposition~\ref{Prop:Uniformize}. We set $$\phi_{n+1}(w) = \theta_n((\phi_n(w), f_{j(n)}), S_n).$$
\vspace{2 mm}

\begin{lemma}
	\label{Lem:PWConvergence}
	For every $w\in W$, the sequence $\phi_n(w)$ is convergent. 
\end{lemma}

\begin{proof}
	Fix $w := ((f_n)_n, (S_n)_n)\in W$, and write $\phi_n(w) = (\alpha_{n,m})_{m< \omega}$, with $\alpha_{n,m}\in \lip$. When $m> n$, we set $\alpha_{n,m}$ to the constant zero function. Fix $m< \omega$ and $g\in G$. We will show that as $n\to \infty$, eventually $\alpha_{n,m}(g)$ is constant. Suppose $k< \omega$ is such that $g\in E_k$. Since $w\in W$, we eventually have that $E_n^3S_n\cap E_k = \emptyset$. This implies that $\phi_{n+1}(w)|_{E_k} = (\phi_n(w), f_{j(n)})|_{E_k}$, so in particular $\alpha_{n,m}(g) = \alpha_{n+1, m}(g)$ whenever $n$ is suitably large.
\end{proof}
\vspace{2 mm}

We can now define $\phi\colon W\to Z$ by setting
$$\phi(w) = \lim_n \phi_n(w)$$
Then $\phi$ is Borel and $G$-equivariant. To see that $\phi$ is injective, suppose $w, w'\in W$, where $w = ((f_n)_n, (S_n)_n)$ and $w' = ((f'_n)_n, (S'_n)_n)$. Find $n< \omega$ and $i< 2$ with $f_{(n,i)}\neq f'_{(n,i)}$. In particular, there is some $k< \omega$ with $f_{(n,i)}|_{E_k}\neq f'_{(n,i)}|_{E_k}$. Because $j\colon \omega\to \omega\times 2$ is an infinite-to-one surjection, we can find $N<\omega$ with $j(N) = (n,i)$ which is as large as desired, in particular large enough so that $E_N^3S_N\cap E_k = \emptyset$ and $E_N^3S'_N\cap E_k = \emptyset$. Writing $\phi(w) = (\alpha_n)_n$ and $\phi(w') = (\alpha'_n)_n$, we have that $\alpha_N|_{E_k} = f_{j(N)}|_{E_k}$ and $\alpha'_N|_{E_k} = f'_{j(N)}|_{E_k}$. Hence $\phi(w)\neq \phi(w')$ as desired.

This concludes the proof that $Z$ is a minimal, model-universal $G$-flow.

\vspace{5 mm}

\noindent
Colin Jahel

\noindent
Universit\'e Claude Bernard - Lyon 1

\noindent
colin.jahel@math.univ-lyon1.fr
\vspace{3 mm}

\noindent
Andy Zucker

\noindent
Universit\'e Claude Bernard - Lyon 1

\noindent
zucker@math.univ-lyon1.fr


\begin{thebibliography}{9}
	\bibitem{AEG}
	S.\ Adams, G.\ A.\ Elliott, and T.\ Giordano, Amenable actions of groups, \emph{Trans.\ Amer.\ Math.\ Soc.}, \textbf{344(2)} (1994), 803--822.
	
	\bibitem{AKL}
	O.\ Angel, A.\ S.\ Kechris, and R.\ Lyons, Random orderings and unique ergodicity of automorphism groups, \emph{J.\ European Math.\ Soc.}, \textbf{16} (2014), 2059--2095.
	
	\bibitem{Ant}
	N.\ Antonyan, Equivariant embeddings and compactifications of free $G$-spaces, \emph{Int.\ J.\ Math.\ Math.\ Sci.}, \textbf{1} (2003), 1--14. 
	
	\bibitem{Elek}
	G.\ Elek, Free minimal actions of countable groups with invariant probability measures, \emph{Ergodic Theory Dyn.\ Sys.}, to appear.
	
	\bibitem{Ellis}
	R.\ Ellis, Lectures on topological dynamics, \emph{W. A. Benjamin, Inc., New York}, 1969.
      
    \bibitem{FHM}
    J. Feldman, P.\ Hahn, C.\ Moore, Orbit structure and countable sections for actions of continuous groups, \emph{Adv.\ Math.}, \textbf{28} (1978), 186--230. 
      
      
	\bibitem{F} 
	G.\ B.\ Folland, \emph{A Course in Abstract Harmonic Analysis}, 2nd ed., CRC Press, 2016.
	
	\bibitem{Kechris}
	A.\ S. \ Kechris, Countable sections for locally compact group actions, \emph{Ergod. Th. \& Dynam. Sys.}, 12 (1992), no. 2, 283-–295.
	
	\bibitem{Mackey}
	G.\ W.\ Mackey, Point realizations of transformation groups, \emph{Illinois J.\ Math.}, \textbf{6} (1962), 327--335.
	
	\bibitem{Slutsky}
	K. \ Slutsky, Lebesgue orbit equivalence of multidimensional Borel flows, \emph{Ergod. Th. \& Dynam. Sys.} 37 (2017), no. 6, 1966–-1996.
	
	\bibitem{V}
	V.\ S.\ Varadarajan, Groups of automorphisms of Borel spaces, \emph{Trans.\ Amer.\ Math.\ Soc.}, \textbf{109} (1963), 191--220. 
	
	\bibitem{W}
	B.\ Weiss, Minimal models for free actions, \emph{Contemporary Mathematics}, \textbf{567} (2012), 249--264
	
	\bibitem{Z}
	A.\ Zucker, A note on minimal models for pmp actions, \emph{Proc.\ Amer.\ Math.\ Soc.}, \textbf{148} (2020), 1161--1168.
\end{thebibliography}
\end{document}